\documentclass[12pt]{amsart}
\usepackage[alphabetic]{amsrefs}
\usepackage{graphicx}
\usepackage{amscd}
\usepackage{upgreek}
\usepackage{stmaryrd}
\usepackage{longtable}
\usepackage[T1]{fontenc}
\usepackage{latexsym, amsmath, amssymb, amsthm}
\usepackage{rsfso}
\usepackage[mathscr]{eucal}
\usepackage{mathtools}
\usepackage{mathptmx}
\usepackage{titletoc}
\usepackage{wrapfig}
\usepackage{float}
\usepackage{xypic}
\usepackage{microtype}
\usepackage{dsfont}
\usepackage{xcolor}
\usepackage{color}
\usepackage[colorlinks,linkcolor=blue,anchorcolor=blue,urlcolor=blue,
citecolor=blue]{hyperref}
\usepackage{hyperref}
\allowdisplaybreaks	
\usepackage[centering, includeheadfoot, hmargin=1.0in, tmargin=0.5in, 
  bmargin=0.6in, headheight=6pt]{geometry}
\newtheorem{theorem}{Theorem}[section]
\newtheorem{prop}[theorem]{Proposition}
\newtheorem{defn}[theorem]{Definition}
\newtheorem{lem}[theorem]{Lemma}
\newtheorem{coro}[theorem]{Corollary}

\newtheorem{thm}[theorem]{Theorem}

\newtheorem{rem}[theorem]{Remark}
\newtheorem{exam}[theorem]{Example}

\newcommand{\ideal}[1]{\ensuremath{\left\langle #1 \right\rangle}}
 
\DeclareMathOperator{\ad}{ad}
\DeclareMathOperator{\aut}{Aut}
\DeclareMathOperator{\Der}{Der}
\DeclareMathOperator{\rank}{rank}
\DeclareMathOperator{\GL}{GL}

\DeclareMathOperator{\iaut}{InnAut}

\DeclareMathOperator{\rad}{rad}

\newcommand{\C}{\mathbb{C}} 
\newcommand{\Q}{\mathbb{Q}} 
\newcommand{\Z}{\mathbb{Z}} 
\newcommand{\N}{\mathbb{N}} 
\newcommand{\FF}{\mathbb{F}} 
\newcommand{\g}{\mathfrak{g}} 
\newcommand{\h}{\mathfrak{h}} 
\newcommand{\p}{\mathfrak{p}} 
\newcommand{\gl}{\mathfrak{gl}} 
\newcommand{\sll}{\mathfrak{sl}} 

\newcommand{\ra}{\longrightarrow} 
 
\newcommand{\B}{\mathcal{B}} 
\newcommand{\D}{\mathcal{D}} 
\newcommand{\HH}{\mathcal{H}} 
 
\newcommand{\V}{\mathcal{V}}

\newcommand{\hbo}{$\hfill\Diamond$} 

\begin{document}
\title{A generalization on derivations of Lie algebras} 
\def\shorttitle{A generalization on derivations of Lie algebras}

\author{\scshape Hongliang Chang}
\address{School of Mathematics and Statistics, Northeast Normal University, Changchun, P.R. China}
\email{changhl023@nenu.edu.cn}

\author{\scshape Yin Chen}
\address{School of Mathematics and Statistics, Northeast Normal University, Changchun, P.R. China}
\email{ychen@nenu.edu.cn}

\author{\scshape Runxuan Zhang}
\address{School of Mathematics and Statistics, Northeast Normal University, Changchun, P.R. China}
\email{zhangrx728@nenu.edu.cn}

\begin{abstract}
We initiate a study on a range of new generalized derivations of finite-dimensional Lie algebras over an algebraically closed field of characteristic zero. This new generalization of derivations has an analogue in the theory of associative prime rings and unites many well-known generalized derivations that have already appeared extensively in the study of Lie algebras and other nonassociative algebras. 
After exploiting fundamental properties, we introduce and analyze their interiors, especially focusing on the rationality of the corresponding Hilbert series. Applying techniques in computational ideal theory we develop an approach to explicitly compute these new generalized derivations for the three-dimensional special linear Lie algebra over the complex field. 
\end{abstract}

\date{\today}
\maketitle \baselineskip=15pt

\dottedcontents{section}[1.16cm]{}{1.8em}{5pt}
\dottedcontents{subsection}[2.00cm]{}{2.7em}{5pt}

\section{Introduction}
\setcounter{equation}{0}
\renewcommand{\theequation}
{1.\arabic{equation}}
\setcounter{theorem}{0}
\renewcommand{\thetheorem}
{1.\arabic{theorem}}

\noindent  As the abstraction of the ratio $\Delta y/\Delta x$ of the differences in calculus, the notion of derivations (or differential operators) stems from the study of algebraic theory of differential equations; see for example \cite{Rit32, vdPS03}. In terms of modern algebra, 
derivations and their generalizations occupy an important place in many areas such as ring theory \cite{Pos57, BK89, Bre93}, (differential) nonassociative algebras \cite{Jac55, Kol73, BKK95, GK08}, differential commutative algebras and algebraic geometry \cite{CGKS02}. Specifically, studying algebras that are endowed with suitable derivations and determining the structure property of an algebraic object by describing its derivations 
both have substantial ramifications in the structure theory of prime rings, differential algebras, and Lie algebras, whereas
finding an efficient way to explicitly characterize derivations and their generalizations  is indispensable in developing connections between algebras and related topics; for example, the Hochschild cohomology
\cite{GL14}. 

Compared with the generalizations of derivations  appeared in \cite{Hva98,LL00,Bre08} 
and motivated by the natural question of whether there exists a more general notion to unite various generalizations of derivations
of a ring or an algebra \cite[Introduction]{DFW18}, we introduce the following notion of $(\alpha,\beta)$-derivations, which connects with the automorphism group when specializing in the case where $\alpha$ and $\beta$ are automorphisms.
\begin{defn}{\rm
Let $A$ be a nonassociative ring (resp. algebra over a commutative ring $R$) and $\alpha,\beta$ be two 
additive mappings (resp. $R$-linear mappings) of $A$. 
An additive mapping (resp. $R$-linear mapping) $D:A\ra A$ is called an $(\alpha,\beta)$-\textbf{derivation} if 
$D(xy) = D(x)\alpha(y)+\beta(x)D(y)$ for all $x,y\in A$.
}\end{defn}
Immediately,  an $(\alpha,\beta)$-derivation is a derivation in the classical sense if and only if $\alpha$ and $\beta$  were both taken to be the identity map. As a relatively general generalization of derivations, $(\alpha,\beta)$-derivations extend many generalizations of derivations, such as $\sigma$-derivations \cite{HLS06}, skew-derivations and classical $(\alpha,\beta)$-derivations \cite{BG97,CL05,DFW18,KP92}, and anti-derivations and $\delta$-derivations \cite{BBC99,Zus10}.
These generalizations have already been shown to be of importance in the study of the deformation theory of algebras, 
differential and generalized functional identities, invariant theory of prime rings and so on. 

Our objective in this article is to give a systematic study on $(\alpha,\beta)$-derivations of finite-dimensional Lie algebras, especially concentrating on that $\alpha$ and $\beta$ both belong to a subgroup of the automorphism group of a Lie algebra. Our point of view is to regard $(\alpha,\beta)$-derivations as points in a suitable affine variety and thus  methods and techniques we use through come from the computational ideal theory.  To articulate the main results, we assume that $\g$ is a finite-dimensional Lie algebra over an algebraically closed field $\FF$ of characteristic zero and $\aut(\g)$ denotes  the automorphism group of $\g$. We write $\Der(\g)$ for the derivation algebra of $\g$ and $\gl(\g)$ for the general linear Lie algebra on $\g$ with the usual bracket product $[D,T]:=D\circ T-T\circ D$ for $D,T\in \gl(\g)$.
\begin{defn}{\rm
Let $G$ be a subgroup of $\aut(\g)$. A linear map $D:\g\ra\g$ is called  a \textbf{$G$-derivation} of $\g$
if there exist two automorphisms $\sigma,\tau\in G$ such that
$D([x,y])=[D(x),\sigma(y)]+[\tau(x),D(y)]$
for all $x,y\in\g$. In this case, $\sigma$ and $\tau$ are called the \textbf{associated} automorphisms of $D$.
}\end{defn}
\noindent We write $\Der_{G}(\g)$ for the set of all $G$-derivations of $\g$.  Fix two automorphisms $\sigma, \tau\in G$, we  denote by $\Der_{\sigma, \tau}(\g)$ the set of all $G$-derivations associated to  $\sigma$ and $\tau$.  Clearly, $\Der_{\sigma, \tau}(\g)\subseteq \Der_{G}(\g)$ is a vector space and  in particular,  $\Der_{1,1}(\g)=\Der(\g)$. 
Note that Remark \ref{rem2.2} below shows that the study of the set $\Der_{\sigma, \tau}(\g)$ with two parameters $\sigma$ and $\tau$ could be reduced to the study of the set $\Der_{\sigma'}(\g)$ with one parameter $\sigma'$. Hence, many statements in the paper are about $G$-derivations associated to one automorphism $\sigma\in G$ and the identity 1 of $G$.
For simplicity of notations, we write $\Der_{\sigma}(\g)$ for $\Der_{\sigma, 1}(\g)$.

\subsection*{Main results and layout}
In Section \ref{sec2}, we mainly explore fundamental  properties of $\Der_{\sigma, \tau}(\g)$ for two fixed automorphisms $\sigma$ and $\tau$. After proving a fact that states that  the analysis of $\Der_{\sigma, \tau}(\g)$ could be reduced to the case of $\Der_{\tau^{-1}\sigma}(\g)$, our first general result shows that  $\Der_{\sigma, \sigma}(\g)$ and  $\Der(\g)$ are isomorphic as Lie algebras.
\begin{thm}\label{thm1.3}
For an arbitrary element  $\sigma$ of $G$, there exists a Lie algebra structure $[-,-]_{\sigma}$ on $\Der_{\sigma,\sigma}(\g)$ such that $\Der_{\sigma,\sigma}(\g)\cong\Der(\g)$ as Lie algebras.
\end{thm}
\noindent Note that
$\Der_{\sigma}(\g)$ is not a Lie subalgebra of $\gl(\g)$ in general. Proposition \ref{prop2.5} gives a sufficient condition to make the direct sum  $\Der(\g)\oplus\Der_{\sigma}(\g)$ of vector spaces to be  a Lie algebra. 

The proof of Proposition \ref{prop2.5} also leads us to introduce and study three kinds of interiors of $\Der_{G}(\g)$ in Section \ref{sec3}. Concentrating on the case where $G$ is infinite cyclic group, our second main result shows that 
the Hilbert series of the big interior $\Der_{G}^{+}(\g)$ of $\Der_{G}(\g)$ is a rational function under some conditions; compared with \cite[Theorem 5.5]{CZ19}. 
\begin{thm}\label{thm1.4}
Let $G=\langle\sigma\rangle$ be an infinite cyclic subgroup of $\aut(\g)$. Suppose there exists $\ell_{0}\in\N^{+}$ and $D\in\Der_{\sigma^{\ell_{0}}}(\g)$  such that $\ad_{D}$ restricts to an invertible map on $\Der_{\sigma^{i}}(\g)$ for all $i\in\Z\setminus\{\ell_{0}\}$.
Then $\HH(\Der_{G}^{+}(\g),t)$ is a rational function.
\end{thm}

Section \ref{sec4} is devoted to investigating comparisons and connections between  $\Der_{\sigma}(\g)$ 
and  other generalizations of derivations of $\g$, such as centroid, quasiderivations, and periodic derivations. 
The third main result gives a sufficient condition such that $\dim(\Der_{\sigma}(\g))$ is less than or equal to $\dim(\g)$ for a centerless Lie algebra $\g$.

\begin{thm}\label{thm1.5}
Let $\g$ be a centerless Lie algebra.
If there exists an element $x_{0}\in\g$ such that $D(x_{0})\neq 0$ for all $D \in\Der_{\sigma}(\g)$, then
$\dim(\Der_{\sigma}(\g))\leqslant\dim(\g)$.
\end{thm}

In Section \ref{sec5}, we apply some techniques from computational ideal theory  to explicitly describe geometric structures of $\Der_{\sigma}(\sll_2(\C))$
for specifically chosen inner automorphisms $\sigma$ of the special linear Lie algebra  $\sll_2(\C)$.
After clarifying the matrix generators of the inner automorphism group of  $\sll_2(\C)$, we take $$\sigma=
  \begin{pmatrix}
    1 & b & -b^2\\
    0 & 1 & -2b\\
     0 & 0 & 1\\
  \end{pmatrix},\quad b\in\C$$
as a sample, and prove the following fourth main result. 

\begin{thm}\label{thm1.6}
As an affine variety,  $\Der_{\sigma}(\sll_2(\C))$ has two irreducible components  $\V(\p_{1})$ and $\V(\p_{2})$, where
 $\V(\p_{1})=\Der(\sll_{2}(\C))$ is three-dimensional and  $\V(\p_{2})$ is a two-dimensional, consisting of all paris $(D,\sigma)$ of the following form:
$$D=\begin{pmatrix}
  0    &  \frac{-a}{2} &\frac{ab}{2} \\
    0  & 0&a\\
    0&0&0
\end{pmatrix}\textrm{ and }\sigma=\begin{pmatrix}
  1 & b & -b^2\\
    0 & 1 & -2b\\
     0 & 0 & 1\\
\end{pmatrix}.$$
\end{thm}
Moreover, if we fix $b$, then $\Der_{\sigma}(\sll_2(\C))$, as a $\C$-vector space, has dimension 4.
Similar methods could applied to the remaining two families of inner automorphisms and the corresponding results will be summarized in Theorems \ref{thm5.11} and \ref{thm5.12}  respectively, without going into detailed proofs.

\subsection*{Conventions and basic concepts} 
Throughout $\g$ is assumed to be a finite-dimensional Lie algebra over an algebraically closed field $\FF$ of characteristic zero and we write $Z(\g)$ for the center of a Lie algebra $\g$ and $I$ denotes the identity map. 
The \textbf{centroid} $C(\g)$  of $\g$ consists of all linear maps $D:\g\ra\g$
such that $[D(x),y]=[x,D(y)]=D([x,y])$ for all $x,y\in\g$.
A linear map $D:\g\ra\g$ is called a \textbf{quasiderivation} of $\g$ if there exists a linear map $T\in \gl(\h)$ satisfying $[D(x),y]+[x,D(y)]=T([x,y])$ for all $x,y\in\g$; see for example \cite{LL00}. A linear map $D:\g\ra\g$ is called an $(\alpha,\beta,\gamma)$-\textbf{derivation} of $\g$ if
there exist $\alpha,\beta,\gamma\in\FF$ such that $\alpha\cdot D[x,y]=\beta\cdot[D(x),y]+\gamma\cdot[x,D(y)]$
for all $x,y\in\g$; see \cite[Section 2.1]{NH08}. We say that $\g$ is a \textbf{perfect} Lie algebra if $[\g,\g]=0.$
Given an element $x\in\g$, we write $Z_{x}(\g)$ for the \textbf{centralizer} of $x$ in $\g$, i.e.,
$Z_{x}(\g)=\{y\in \g\mid [x,y]=0\}.$ We say that  $D\in\Der_{\sigma}(\g)$  is \textbf{periodic} if there exists a positive integer $m\in\N^{+}$ such that $D^{m}=I$. The minimum $m$ such that $D^{m}=I$ is called the \textbf{order} of $D$ and denoted by $|D|=m$.

Given $n\in\N^+$ and an ideal $J$ of a polynomial ring over $\FF$ in $n$ variables. The zero set 
$\V(J)$ consisting of all common roots in $\FF^n$ of all polynomials of $J$ is called the \textbf{affine variety} of $J$.
Note that $\V(J)=\V(\rad(J))$, where $\rad(J)$ denotes the radical ideal of $J$. As a topological space (Zariski's topology), 
$\V(J)$ is irreducible if and only if $J$ is a prime ideal. See for example \cite[Chapter 4]{CLO07} for more details. 

\subsection*{Acknowledgements} We thank the two referees for  careful readings of the manuscript and for a number of constructive corrections and suggestions. This research was partially supported by NSFC (No. 11301061).

\section{Fundamental Properties} \label{sec2}
\setcounter{equation}{0}
\renewcommand{\theequation}
{2.\arabic{equation}}
\setcounter{theorem}{0}
\renewcommand{\thetheorem}
{2.\arabic{theorem}}

\noindent This section contains fundamental results about $G$-derivations associated with two fixed automorphisms in $G$, where  $G$ always denotes a subgroup of $\aut(\g)$. Given two arbitrary elements  $\sigma, \tau\in G$, let's recall that
$\Der_{\sigma, \tau}(\g)=\big\{D:\g\ra\g\mid D([x,y])=[D(x),\sigma(y)]+[\tau(x),D(y)],\textrm{ for all }x,y\in\g\big\}.$

\begin{prop} \label{prop2.1}
$\dim_{\FF}(\Der_{\sigma,\tau}(\g))=\dim_{\FF}(\Der_{\tau^{-1}\sigma}(\g))$.
\end{prop}

\begin{proof}
One may define a map $\varphi_{\tau}:\Der_{\sigma,\tau}(\g)\ra \Der_{\tau^{-1}\sigma}(\g)$ by $D\mapsto \tau^{-1}\circ D$. In fact,
for arbitrary $x,y\in\g$, since $D([x,y])=[D(x),\sigma(y)]+[\tau(x),D(y)]$, it follows that
$$\tau^{-1}\circ D([x,y])=[\tau^{-1}\circ D(x),\tau^{-1}\sigma(y)]+[x,\tau^{-1}\circ D(y)],$$
which means that $\tau^{-1}\circ D\in  \Der_{\tau^{-1}\sigma}(\g)$. Clearly, $\tau^{-1}\circ (D_{1}+D_{2})=\tau^{-1}\circ D_{1}+\tau^{-1}\circ D_{2}$ and $\tau^{-1}\circ (a\cdot D)=a\cdot\tau^{-1}\circ D$ for all $D,D_{1},D_{2}\in \Der_{\sigma,\tau}(\g)$ and
$a\in\FF$. Thus $\varphi_{\tau}$ is a linear map. Similarly, the
map $\phi_{\tau}:\Der_{\tau^{-1}\sigma}(\g)\ra\Der_{\sigma,\tau}(\g)$ defined by $D\mapsto \tau\circ D$, is also linear, and
 $\varphi_{\tau}^{-1}=\phi_{\tau}$. Hence, $\varphi_{\tau}$ is a linear isomorphism and
$\dim_{\FF}(\Der_{\sigma,\tau}(\g))=\dim_{\FF}(\Der_{\tau^{-1}\sigma}(\g))$.
\end{proof}

\begin{rem}\label{rem2.2}
{\rm
This result indicates that the study of $\Der_{\sigma,\tau}(\g)$ with two parameters $\sigma,\tau$ actually could be turned to
the study of  $\Der_{\sigma'}(\g)$ with one parameter $\sigma'\in G$. In particular, when $\sigma=\tau$, we see that
$\Der_{\sigma,\sigma}(\g)$ is isomorphic to $\Der(\g)$ as vector spaces; our first general result, Theorem \ref{thm1.3}, even shows that
this linear isomorphism can be extended to an isomorphism of Lie algebras.
\hbo}\end{rem}

To show  Theorem \ref{thm1.3}, given $\sigma\in G$, we define a bracket product  $[-,-]_{\sigma}$ on $\Der_{\sigma,\sigma}(\g)$ as follows:
$$[D,T]_{\sigma}:=\varphi_{\sigma}^{-1}([\varphi_{\sigma}(D),\varphi_{\sigma}(T)])$$
where $D,T\in \Der_{\sigma,\sigma}(\g)$. 

\begin{proof}[Proof of Theorem \ref{thm1.3}]
As the proof of Proposition \ref{prop2.1}, the linear map $\varphi_{\sigma}:\Der_{\sigma,\sigma}(\g)\ra \Der(\g)$ defined by $D\mapsto \sigma^{-1}\circ D$, gives rise to a bijective linear map between $\Der_{\sigma,\sigma}(\g)$ and $\Der(\g)$.
Using the fact that $\Der(\g)$ is a Lie algebra and
$\varphi_{\sigma}^{-1}$ is linear, we observe that $(\Der_{\sigma,\sigma}(\g),[-,-]_{\sigma})$ is a Lie algebra.
To complete the proof, it suffices to show that $\varphi_{\sigma}$ is a Lie homomorphism. In fact,
$\varphi_{\sigma}([D,T]_{\sigma})=\varphi_{\sigma}(\varphi_{\sigma}^{-1}([\varphi_{\sigma}(D),\varphi_{\sigma}(T)]))=
[\varphi_{\sigma}(D),\varphi_{\sigma}(T)]$ and so $\varphi_{\sigma}$ is a Lie homomorphism, as desired. 
\end{proof}

\begin{coro}
Let  $D\in\Der_{\sigma,\sigma}(\g)$ be an invertible map. Then
$$x\ast y:=D^{-1}([\sigma(x),D(y)])$$
gives rise to a left-symmetric algebra structure on $\g$.
\end{coro}

\begin{proof}
Recall that for an invertible derivation $T\in\Der(\g)$ and $x,y\in\g$, the product
$$x\ast y=T^{-1}([x,T(y)])$$
induces a left-symmetric algebra on $\g$; see \cite[Proposition 3.3]{Bur06}. Taking $T=\varphi_{\sigma}(D)$, we see that
$x\ast y=\varphi_{\sigma}(D)^{-1}([x,\varphi_{\sigma}(D)(y)])
=D^{-1}\circ \sigma([x,\sigma^{-1}\circ D(y)])
= D^{-1}([\sigma(x),D(y)])$
as desired.
\end{proof}

Note that unlike the derivation algebra $\Der(\g)$, the vector space $\Der_{\sigma}(\g)$ is in general not a Lie subalgebra of $\gl(\g)$. The following result shows that under some conditions, $\Der_{\sigma}(\g)$ and $\Der(\g)$ could coincide; 
and the direct sum $\Der(\g)\oplus\Der_{\sigma}(\g)$ of vector spaces might be a Lie algebra.

\begin{prop}\label{prop2.4}
 If $(\sigma-\tau)(\g)$ is contained in $Z(\g)$, then $\Der_{\sigma}(\g)=\Der_{\tau}(\g)$. In particular, if $(\sigma-I)(\g)\subseteq Z(\g)$, then $\Der_{\sigma}(\g)=\Der(\g)$  is a Lie subalgebra of $\gl(\g)$.
\end{prop}

\begin{proof} 
As $\sigma(y)-\tau(y)\in Z(\g)$ for all $y\in\g$, we have $[D(x),\sigma(y)]=[D(x),\tau(y)]$ for all $x\in\g$ and $D\in\gl(\g)$. In particular, if $D\in \Der_{\sigma}(\g)$, then $D([x,y])=[D(x),\sigma(y)]+[x,D(y)]=[D(x),\tau(y)]+[x,D(y)]$. This means that 
$D\in \Der_{\tau}(\g)$ and $\Der_{\sigma}(\g)\subseteq \Der_{\tau}(\g)$. Switching the roles of $\sigma$ and $\tau$, we see that
$\Der_{\sigma}(\g)=\Der_{\tau}(\g)$.
\end{proof}

\begin{prop}\label{prop2.5}
Let $\sigma$ be an involutive automorphism of $\g$ that commutes with each element of $\Der(\g)$ and $\Der_{\sigma}(\g)$.
Then $\Der(\g)\oplus\Der_{\sigma}(\g)$ is a Lie algebra.
\end{prop}

\begin{proof} Note that $\Der(\g)\oplus\Der_{\sigma}(\g)$ is a vector space, thus it suffices to show that the bracket product of arbitrary two elements in $\Der(\g)\oplus\Der_{\sigma}(\g)$ are closed, i.e.,
we need to show that $[D,T]=D\circ T-T\circ D\in \Der(\g)$ or $\Der_{\sigma}(\g)$
for all $D\in \Der_{\sigma^{k}}(\g)$ and $T\in \Der_{\sigma^{\ell}}(\g)$ where $k,\ell\in\{0,1\}$.  Given arbitrary elements $x,y\in\g$, we have
\begin{eqnarray*}
D\circ T([x,y])& = & D([T(x),\sigma^{\ell}(y)]+[x,T(y)])  \\
 & = & [D\circ T(x),\sigma^{k+\ell}(y)]+[T(x),D\circ\sigma^{\ell}(y)]+[D(x),\sigma^{k} \circ T(y)]+[x,D\circ T(y)]
\end{eqnarray*}
and
$T\circ D([x,y])= [T\circ D(x),\sigma^{\ell+k}(y)]+[D(x),T\circ\sigma^{k}(y)]+[T(x),\sigma^{\ell}\circ D(y)]+[x,T\circ D(y)]$.
Note that $\sigma$ commutes with every element of  $\Der_{\sigma}(\g)$, thus
\begin{eqnarray*}
(D\circ T-T\circ D)([x,y]) &=&[[D,T](x),\sigma^{\ell+k}(y)]+[x,[D,T](y)]+[D(x),(\sigma^{k} \circ T-T\circ\sigma^{k})(y)]\\
&&+[T(x),(D\circ\sigma^{\ell}-\sigma^{\ell} \circ D)(y)]\\
&=&[[D,T](x),\sigma^{\ell+k}(y)]+[x,[D,T](y)].
\end{eqnarray*}
As $\sigma^{k+\ell}=\sigma$ or $I$, it follows that 
$[D,T]$ is either in $\Der(\g)$ or $\Der_{\sigma}(\g)$.
Hence, $\Der(\g)\oplus\Der_{\sigma}(\g)$ is a Lie algebra.
\end{proof}

We close this section with the following result that verifies the commutativity of $D\in \Der_{\sigma}(\g)$ and $\sigma$ when $\g$ is a perfect Lie algebra.

\begin{prop}
Let $\g$ be a nonabelian Lie algebra and $D\in \Der_{\sigma}(\g)$ be an element such that $[D,\sigma](\g)\subseteq Z(\g)$.
Then $[\g,\g]$ is contained in the kernel of $[D,\sigma]$. In this case, if $\g$ is perfect, then
$D$ commutes with $\sigma$.
\end{prop}

\begin{proof}
For $x,y\in\g$, we note that $(\sigma\circ D)([x,y])=[\sigma\circ D(x),\sigma^{2}(y)]+[\sigma(x),\sigma\circ D(y)]$
and $(D\circ \sigma)([x,y])=[D\circ\sigma(x),\sigma^{2}(y)]+[\sigma(x),D\circ \sigma(y)]$. Thus
\begin{eqnarray*}
[D,\sigma]([x,y])&=&(D\circ\sigma-\sigma\circ D)[x,y]\\
&=&[[D,\sigma](x),\sigma^{2}(y)]+[\sigma(x),[D,\sigma](y)]\\
&=&0.
\end{eqnarray*}
The last equation follows from the assumption that $[D,\sigma](\g)\subseteq Z(\g)$. Thus $[\g,\g]$ is contained in the kernel of $[D,\sigma]$. Moreover, if $\g$ is perfect, then $\g=[\g,\g]\subseteq\ker([D,\sigma])\subseteq\g$. Thus $\ker([D,\sigma])=\g$
and $[D,\sigma]=0$.
\end{proof}

\section{The Interiors of $G$-derivations}\label{sec3}
\setcounter{equation}{0}
\renewcommand{\theequation}
{3.\arabic{equation}}
\setcounter{theorem}{0}
\renewcommand{\thetheorem}
{3.\arabic{theorem}}

\noindent For those $G$-derivations associated with an involutive automorphism $\sigma\in G$, being able to commute with $\sigma$ is important for embedding $\Der_{\sigma}(\g)$ into a larger Lie algebra; see the proof of Proposition \ref{prop2.5}.
To better understand the structures of $\Der_{\sigma}(\g)$ and $\Der_{G}(\g)$, we specialize in several specific subspaces of $\Der_{\sigma}(\g)$ which we call the \textbf{interiors} of $\Der_{G}(\g)$. We also study the rationality of the Hilbert series for the direct sum of these subspaces when $G$ is a cyclic group.

We write $\Der_{\sigma}^{+}(\g)$ for the subset of $\Der_{\sigma}(\g)$ consisting of elements that commute with
$\sigma$ and write $\Der_{\sigma}^{-}(\g)$ for the subset of $\Der_{\sigma}(\g)$ consisting of elements that commute with each
$\tau\in G$.  Clearly, $\Der_{\sigma}^{-}(\g)\subseteq \Der_{\sigma}^{+}(\g)\subseteq\Der_{\sigma}(\g)$.
Note that $\Der_{G}(\g)$ is not a vector space in general; however, for each $\sigma\in G$,
$\Der_{\sigma}^{+}(\g), \Der_{\sigma}^{-}(\g)$ and $\Der_{\sigma}(\g)$, as subsets of $\Der_{G}(\g)$, are vector spaces.
Thus it is natural to consider some kinds of ``sum'' of all $\Der_{\sigma}^{\pm}(\g)$ or all  $\Der_{\sigma}(\g)$, when $\sigma$ runs through the whole group $G$. We define
$$\Der_{G}^{+}(\g):=\bigoplus_{\sigma\in G}\Der_{\sigma}^{+}(\g)\textrm{ and }\Der_{G}^{-}(\g):=\bigoplus_{\sigma\in G}\Der_{\sigma}^{-}(\g),$$
which are called the \textbf{big interior} and the \textbf{small interior} of $\Der_{G}(\g)$ respectively.
We also define
$$\Der_{G}^{\circ}(\g):=\bigoplus_{\sigma\in G}\Der_{\sigma}(\g)$$
and call it the \textbf{interior} of $\Der_{G}(\g)$. Clearly,
$\Der_{G}^{-}(\g)\subseteq \Der_{G}^{+}(\g)\subseteq \Der_{G}^{\circ}(\g)$.

\begin{exam}{\rm
Consider the case where $G=\{1\}$ is the trivial group.
As $\Der_{G}^{-}(\g)=\Der_{G}(\g)=\Der(\g)$, it follows that
$\Der_{G}^{-}(\g)=\Der_{G}^{+}(\g)=\Der_{G}^{\circ}(\g)=\Der_{G}(\g)=\Der(\g)$.
\hbo}\end{exam}

\begin{exam}\label{exam3.2}
{\rm
Consider the case where $G$ is cyclic generated by $\sigma$. Then $\Der_{G}^{+}(\g)=\Der_{G}^{-}(\g)$ and for $\ast\in\{+,-,\circ\}$,
$$\Der_{G}^{\ast}(\g)=\Der_{\langle\sigma\rangle}^{\ast}(\g):=\begin{cases}
   \bigoplus_{k\in\Z}\Der_{\sigma^{k}}^{\ast}(\g),   & \text{if $|G|=\infty$ }, \\
   \bigoplus_{k\in\{0,1,\dots,m-1\}}\Der_{\sigma^{k}}^{\ast}(\g),   & \text{if $|G|=m<\infty$}.
\end{cases}$$
where $\sigma^{0}=1$, $\sigma^{1}=\sigma, \sigma^{k}=\sigma^{k-1}\circ\sigma$, and we make the convention that
$\Der_{\sigma^{k}}^{\circ}(\g):=\Der_{\sigma^{k}}(\g)$.
In this case, $\Der_{\langle\sigma\rangle}^{\ast}(\g)$ is a $\Z$-graded vector space and recall that the \textbf{Hilbert series} of
$\Der_{\langle\sigma\rangle}^{\ast}(\g)$ is defined as follows:
$$\HH(\Der_{\langle\sigma\rangle}^{\ast}(\g),t):=\sum_{k\in\Z} \dim_{\C}(\Der_{\sigma^{k}}^{\ast}(\g))\cdot t^{k}.$$
Clearly, if $\sigma$ is of finite order, $\HH(\Der_{\langle\sigma\rangle}^{\ast}(\g),t)$
is a rational function in $\Q[t]$.
\hbo}\end{exam}

\begin{prop}\label{prop3.3}
If $G$ is an abelian group, then $\Der_{G}^{-}(\g)$ is a Lie algebra with the usual bracket product.
\end{prop}

\begin{proof}
As $\Der_{G}^{-}(\g)$ is a vector space, it suffices to show that $[D,T]=D\circ T-T\circ D\in \Der_{G}^{-}(\g)$ for every $D,T\in \Der_{G}^{-}(\g)$. We may suppose $D\in \Der_{\sigma}^{-}(\g)$ and $T\in \Der_{\tau}^{-}(\g)$. Then for $x,y\in\g$, we see that
$D\circ T([x,y])=[D\circ T(x),\sigma\tau(y)]+[T(x),D\circ\tau(y)]+[D(x),\sigma\circ T(y)]+[x,D\circ T(y)]$ and
$T\circ D([x,y])=[T\circ D(x),\tau\sigma(y)]+[D(x),T\circ\sigma(y)]+[T(x),\tau\circ D(y)]+[x,T\circ D(y)]$. Note that
$G$ is abelian, thus
$$[D,T]([x,y])=(D\circ T-T\circ D)([x,y])=[[D,T](x),\sigma\tau(y)]+[x,[D,T](y)],$$
which, together with the immediate fact that $[D,T]$ commutes with every automorphism in $G$, implies that
$[D,T]\in \Der_{\sigma\tau}^{-}(\g)\subseteq  \Der_{G}^{-}(\g)$.
\end{proof}

\begin{coro}\label{coro3.4}
Let $G$ be a cyclic group. Then $\Der_{G}^{+}(\g)=\Der_{G}^{-}(\g)$ are  Lie algebras with the usual bracket product.
\end{coro}

\begin{rem}{\rm
Proposition \ref{prop2.5} could be regarded as a special case of Corollary \ref{coro3.4} when $G=\ideal{\sigma}$ is a cyclic group of order 2.
\hbo}\end{rem}

Comparing with Proposition \ref{prop2.4} that gives a sufficient condition such that $\Der_{\sigma}(\g)$ and $\Der_{\tau}(\g)$ coincides, we consider another extreme case where their intersection is zero.

\begin{prop}\label{prop3.6}
Let $\sigma,\tau\in G$ be two automorphisms of $\g$ such that $(\sigma^{-1}\tau)(w)\notin Z_{w}(\g)$ for all nonzero $w\in \g$.
Then $\Der_{\sigma}(\g) \cap \Der_{\tau}(\g)=0$.
\end{prop}

\begin{proof}
Assume by the way of contradiction that there exists a nonzero $D\in \Der_{\sigma}(\g) \cap \Der_{\tau}(\g)$.
Then there exists an element $x_{0}\in\g$ such that $D(x_{0})\neq 0$. As $D\in \Der_{\sigma}(\g) \cap \Der_{\tau}(\g)$,
thus for all $x,y\in\g$, it follows that $[D(x),\sigma(y)]=[D(x),\tau(y)]$, which is equivalent to
$$
[\sigma^{-1}\circ D(x),y-(\sigma^{-1}\tau)(y)]=0.
$$
Setting $x=x_{0}$ and $y=y_{0}:=(\sigma^{-1}\circ D)(x_{0})$ in previous equation, we see that
$$[y_{0},y_{0}-(\sigma^{-1}\tau)(y_{0})]=0.$$
Thus $[y_{0},(\sigma^{-1}\tau)(y_{0})]=0$, i.e., $(\sigma^{-1}\tau)(y_{0})\in Z_{y_{0}}(\g)$.
On the other hand, the assumption that $(\sigma^{-1}\tau)(w)\notin Z_{w}(\g)$ for all nonzero $w\in \g$, forces $y_{0}$ must to be zero. However, since $D(x_{0})\neq 0$ and $\sigma^{-1}$ is bijective, it follows that $y_{0}\neq 0$.
This contradiction shows that $\Der_{\sigma}(\g) \cap \Der_{\tau}(\g)=0$.
\end{proof}

\begin{rem}{\rm
The referee asked whether the converse statement of  Proposition \ref{prop3.6} remains to follow. Actually, it is not true in general. We may use some results appeared in Section \ref{sec5} below to give a counterexample. Let $\g=\sll_2(\C)$. We take $\tau=\exp(D_1)$ in Theorem \ref{thm5.3} and $\sigma=1$.  Together Theorem \ref{thm5.1} and Corollary \ref{coro5.6} imply that $\Der_{\sigma}(\g) \cap \Der_{\tau}(\g)=0$. On the other hand, since $\tau$ fixes $e_1$, it follows that $[\sigma^{-1}\tau(e_1),e_1]=[\tau(e_1),e_1]=[e_1,e_1]=0$, i.e., $(\sigma^{-1}\tau)(e_1)\in Z_{e_1}(\g)$.
\hbo}\end{rem}

The above Proposition \ref{prop3.6} indicates the set $\Der_{G}(\g)$ might be very complicated and hard to describe.
Hence a relatively reasonable way to understand $\Der_{G}(\g)$ is to analyze the interiors of  $\Der_{G}(\g)$. As we have seen that if $G$ is an infinite cyclic group, the interiors of  $\Der_{G}(\g)$ are $\Z$-graded vector spaces. 
It is well known that the Hilbert series of a $\Z$-graded vector space
is an important invariant that encodes information on the dimensions of the subspaces into an infinite series; see for example \cite[Section 5]{CZ19}.

We are ready to prove our second main result.

\begin{proof}[Proof of Theorem \ref{thm1.4}]
By the proof of Proposition \ref{prop3.3} we see that $\ad_{D}:\Der_{\sigma^{k}}(\g)\ra \Der_{\sigma^{k+\ell_{0}}}(\g)$
is a linear isomorphism for all $k\in\Z\setminus\{\ell_{0}\}$. Hence, 
$$\dim(\Der_{\sigma^{k}}(\g))=\dim(\Der_{\sigma^{k+\ell_{0}}}(\g))=\dim(\Der_{\sigma^{k-\ell_{0}}}(\g))$$ for each $k\in\N\setminus\{\ell_{0}\}$. Moreover,
$\dim(\Der_{\sigma^{\ell_{0}}}(\g))=\dim(\Der(\g))=\dim(\Der_{\sigma^{-\ell_{0}}}(\g))=\cdots$ and
$\dim(\Der_{\sigma^{2\ell_{0}}}(\g))=\dim(\Der_{\sigma^{3\ell_{0}}}(\g))=\cdots$.
Note that
$$\HH(\Der_{G}^{+}(\g),t) = \sum_{k=\ell_{0}+1}^{\infty} \dim(\Der_{\sigma^{k}}(\g))\cdot t^{k}+\sum_{k=\ell_{0}}^{-\infty} \dim(\Der_{\sigma^{k}}(\g))\cdot t^{k}$$
and
\begin{eqnarray*}
\sum_{k=\ell_{0}+1}^{\infty} \dim(\Der_{\sigma^{k}}(\g))\cdot t^{k}& = & t^{\ell_{0}+1} \left(\frac{m_{0}}{1-t^{\ell_{0}}}+  \frac{m_{1}\cdot t}{1-t^{\ell_{0}}}+\cdots+ \frac{m_{\ell_{0}-1}\cdot t^{\ell_{0}-1}}{1-t^{\ell_{0}}}\right)\\
& = & \frac{t^{\ell_{0}+1}\sum_{i=0}^{\ell_{0}-1}m_{i}\cdot t^{i}}{1-t^{\ell_{0}}}
\end{eqnarray*}
where $m_{i}=\dim(\Der_{\sigma^{\ell_{0}+1+i}}(\g))$ for $0\leqslant i\leqslant \ell_{0}-1$. Similarly,
\begin{eqnarray*}
\sum_{k=\ell_{0}}^{-\infty} \dim(\Der_{\sigma^{k}}(\g))\cdot t^{k}& = & \sum_{k=1}^{\ell_{0}} \dim(\Der_{\sigma^{k}}(\g))\cdot t^{k}+\sum_{k=0}^{-\infty} \dim(\Der_{\sigma^{k}}(\g))\cdot t^{k}\\ 
& = & \sum_{k=1}^{\ell_{0}} \dim(\Der_{\sigma^{k}}(\g))\cdot t^{k}+\sum_{k=0}^{\infty} \dim(\Der_{\sigma^{-k}}(\g))\cdot t^{-k}\\
&=&\textrm{ a polynomial in } t + \textrm{ a rational function in } t^{-1} 
\end{eqnarray*}
is also a rational function. Hence,
$\HH(\Der_{G}^{+}(\g),t)$ is a rational function, as desired.
\end{proof}

\section{Connections with (Generalized) Derivations} \label{sec4}
\setcounter{equation}{0}
\renewcommand{\theequation}
{4.\arabic{equation}}
\setcounter{theorem}{0}
\renewcommand{\thetheorem}
{4.\arabic{theorem}}

\noindent In this section, we study comparisons and connections between $G$-derivations and some well-knowns (generalized) derivations of a Lie algebra $\g$, such as centroid, quasiderivations, $(\alpha,\beta,\gamma)$-derivations, and periodic derivations. Recall that the centroid $C(\g)$ is a Lie subalgebra of $\gl(\g)$ and we denote by $\ad:\g\ra\gl(\g)$  the adjoint map that sends $x$ to $\ad_{x}$.

\begin{prop}\label{prop4.1}
For $\sigma\in G$ and $D\in C(\g)\cap\Der_{\sigma}(\g)$, we have $\ad_{D(x)}=0$ for all $x\in \g$. In particular,
if $\g$ is centerless, then $C(\g)\cap\Der_{\sigma}(\g)=0.$
\end{prop}

\begin{proof}
Since $D\in C(\g)\cap\Der_{\sigma}(\g)$, it follows that $[D(x),\sigma(y)]=0$ for all $x,y\in\g$.
As $\sigma$ is bijective, $D(x)\in Z(\g)=\ker(\ad)$, i.e., $\ad_{D(x)}=0$.  In particular, if $Z(\g)$ is zero,
then the map $\ad$ is injective and it has a left inverse. Hence,  the fact that $\ad_{D(x)}=0$ for all $x\in \g$ implies that $D=0$.
\end{proof}

The following example indicates that if $\g$ has a nontrivial center, then the intersection of $C(\g)$ and $\Der_{\sigma}(\g)$ is not necessarily zero.

\begin{exam}{\rm
Consider the three-dimensional Lie algebra $\g$ defined by a basis $\{e_{1},e_{2},e_{3}\}$ and the unique nontrivial commutator relation $[e_{1},e_{2}]=e_{3}$. Let $D\in C(\g), \sigma\in G$, and $T\in \Der_{\sigma}(\g)$
be arbitrary elements. With respect to the basis $\{e_{1},e_{2},e_{3}\}$, we may suppose that
$$D=(a_{ij})^{t}_{3\times3}, \sigma=(b_{ij})^{t}_{3\times3},\textrm{ and }T=(c_{ij})^{t}_{3\times3}.$$
Using  the relations $D([e_{i},e_{j}])=[D(e_{i}),e_{j}]=[e_{i},D(e_{j})]$
for all $i,j\in\{1,2,3\}$, we observe that $a_{32}=a_{31}=0, a_{11}=a_{22}=a_{33}$ and thus conclude that
$$C(\g)=\left\{\begin{pmatrix}
   a & b & 0  \\
   c & a & 0 \\
   d & e & a
\end{pmatrix}: a,b,c,d,e\in\FF\right\}.$$
Similarly, the fact that $\sigma([e_{i},e_{j}])=[\sigma(e_{i}),\sigma(e_{j})]$ for all $i,j\in\{1,2,3\}$ implies that $b_{31}=b_{32}=0$ and $b_{33}=b_{11}b_{22}-b_{12}b_{21}$. Thus
$\sigma=\bigl(\begin{smallmatrix}
   B   & 0   \\
    B'  & \det(B)
\end{smallmatrix}\bigr)$
for some $B\in\GL(2,\FF)$ and $(B')^{t}\in \FF^{2}$. In particular,  we may take $B'=0$ and $B=\bigl(\begin{smallmatrix}
     1 & -1   \\
      0&1
\end{smallmatrix}\bigr)$ for simplicity. Putting above facts all together into
$$
T([e_{i},e_{i}])=[T(e_{i}),\sigma(e_{j})]+[e_{i},T(e_{j})]
$$
where $i,j\in\{1,2,3\}$, we see that
 $c_{31}=c_{32}=0$ and $c_{33}=c_{11}+c_{22}=c_{22}+c_{21}+c_{11}$. Hence,
 $$\Der_{\sigma}(\g)=\left\{\begin{pmatrix}
   a & b & 0  \\
   0 & c & 0 \\
   d & e & a+c
\end{pmatrix}: a,b,c,d,e\in\FF\right\}.$$
Clearly, $$C(\g)\cap \Der_{\sigma}(\g)=\left\{\begin{pmatrix}
   0 & b & 0  \\
   0 & 0 & 0 \\
   d & e & 0
\end{pmatrix}: b,d,e\in\FF\right\}$$
is a three-dimensional vector space.
\hbo}\end{exam}

We consider a Lie subalgebra $\h$ of $\g$ and an automorphism $\sigma$ of $\g$ with $\sigma(\h)\subseteq\h$.
Denote by $\Der_{\sigma,\h}(\g)$ the set of all $\sigma$-derivations of $\g$ that stabilizes $\h$, i.e.,
$$\Der_{\sigma,\h}(\g)=\{D\in\Der_{\sigma}(\g)\mid D(\h)\subseteq \h\}.$$

\begin{prop}
The set $\Der_{\sigma,\h}(\g)$ is a subspace of $\Der_{\sigma}(\g)$.
Moreover, if $\h$ is an ideal of $\g$ with $[\h,\h]=\h$, then $\Der_{\sigma,\h}(\g)=\Der_{\sigma}(\g)$.
\end{prop}

\begin{proof}
For $a,b\in\FF$ and $D,T\in \Der_{\sigma,\h}(\g)$, we have $(aD+bT)(\g)=aD(\g)+bT(\g)\subseteq D(\g)+T(\g)\subseteq \g$.
Thus $\Der_{\sigma,\h}(\g)$ is a subspace of $\Der_{\sigma}(\g)$. Further, for the second assertion,
it suffices to show that $\Der_{\sigma}(\g)\subseteq \Der_{\sigma,\h}(\g)$. In fact, take arbitrary $D\in \Der_{\sigma}(\g)$
and $x\in\h$. Since $[\h,\h]=\h$, we write $x=[y,z]$ for some $y,z\in\h$. The fact that $\h$ is stabilized by $\sigma$ implies that
$$D(x)=D([y,z])=[D(y),\sigma(z)]+[y,D(z)]\in\h.$$
Hence, $D(\h)\subseteq\h$ and the second assertion follows.
\end{proof}

As $\h$ is $\sigma$-stable, we see that the automorphism $\sigma$ restricts to an automorphism $\sigma|_{\h}$ of $\h$, and so
an arbitrary element $D\in \Der_{\sigma,\h}(\g)$ also restricts to a $\sigma|_{\h}$-derivation of $\h$.
Thus the restriction map induces a natural linear map:
$$\alpha:\Der_{\sigma,\h}(\g)\ra \Der_{\sigma|_{\h}}(\h),~ D\mapsto D|_{\h}.$$
For $D\in \Der_{\sigma,\h}(\g)$, we define a map $\widetilde{D}:[\g,\g] \ra\g$ by
$$
[x,y]\mapsto 2D([x,y])+[D(y),\sigma(x)]+[\sigma(y),D(x)].
$$

\begin{prop}
Suppose there exists $x_{0}\in\g$ such that $\ad_{x_{0}}\in\aut(\h)$. Then
$\widetilde{D}$ restricts to a linear map $\widetilde{D}|_{\h}:\h\ra\h$.
\end{prop}

\begin{proof}
In fact, since $\h=\ad_{x_{0}}(\h)=[x_{0},\h]\subseteq[\g,\g]$, for every $x\in\h$, there exists $y\in\h$ such that
$x=[x_{0},y]$ and $\widetilde{D}(x)=\widetilde{D}([x_{0},y])=2D([x_{0},y])+[D(y),\sigma(x_{0})]+[\sigma(y),D(x_{0})]
\in D(\h)+[D(\h),\h]+[\h,D(\h)]\subseteq \h$. Thus $\widetilde{D}(\h)\subseteq\h$.
To see the restriction map $\widetilde{D}|_{\h}$ is linear, we may assume a further element $x'\in\h$ with
$x'=[x_{0},y']$ for some $y'\in\h$. Then
\begin{eqnarray*}
\widetilde{D}(ax+bx')& = & \widetilde{D}([x_{0},ay]+[x_{0},by']) \\
 & = & \widetilde{D}([x_{0},ay+by']) \\
 &=& 2D([x_{0},ay+by'])+[D(ay+by'),\sigma(x_{0})]+[\sigma(ay+by'),D(x_{0})]\\
 &=&a\widetilde{D}(x)+b\widetilde{D}(x')
\end{eqnarray*}
which implies that $\widetilde{D}$ is linear.
\end{proof}

We write $\textrm{QDer}(\h)$ for the set of all quasiderivations of $\h$.

\begin{prop}
Suppose there exists $x_{0}\in\g$ such that $\ad_{x_{0}}\in\aut(\h)$ and $D(x_{0})\in Z(\h)$ for all $D\in \Der_{\sigma,\h}(\g)$. Then
$\alpha(\Der_{\sigma,\h}(\g))\subseteq {\rm QDer}(\h).$
\end{prop}

\begin{proof} We take arbitrary elements $x,y\in\h$ and $D\in \Der_{\sigma,\h}(\g)$.
Since $D(x_{0})$ is a central element of $\g$, it follows that $D\circ \ad_{x_{0}}(x)=D([x_{0},x])=[x_{0},D(x)]=
\ad_{x_{0}}\circ D(x)$, which means that $D\circ \ad_{x_{0}}=\ad_{x_{0}}\circ D$.
Moreover, note that
\begin{eqnarray*}
\widetilde{D}|_{\h}\circ\ad_{x_{0}}([x,y])& = & \widetilde{D}([\ad_{x_{0}}(x),\ad_{x_{0}}(y)]) \\
 & = & 2D([\ad_{x_{0}}(x),\ad_{x_{0}}(y)])+[D\circ \ad_{x_{0}}(y),\sigma\circ \ad_{x_{0}}(x)]\\
 &&+[\sigma\circ \ad_{x_{0}}(y),D\circ \ad_{x_{0}}(x)]\\
 &=& D([\ad_{x_{0}}(x),\ad_{x_{0}}(y)])-[D\circ \ad_{x_{0}}(x),\sigma\circ \ad_{x_{0}}(y)]\\
 &&-D([\ad_{x_{0}}(y),\ad_{x_{0}}(x)])+[D\circ \ad_{x_{0}}(y),\sigma\circ \ad_{x_{0}}(x)]\\
 &=& [D\circ \ad_{x_{0}}(x),\ad_{x_{0}}(y)]+[\ad_{x_{0}}(x),D\circ \ad_{x_{0}}(y)]\\
 &=& \ad_{x_{0}}([D(x),y]+[x,D(y)]).
\end{eqnarray*}
Thus, $[D(x),y]+[x,D(y)]=\ad_{x_{0}}^{-1}\circ \widetilde{D}|_{\h}\circ\ad_{x_{0}}([x,y])$.
Clearly, $D$ restricts to a quasiderivation of $\h$ and we obtain $\alpha(\Der_{\sigma,\h}(\g))\subseteq {\rm QDer}(\h).$
\end{proof}

\begin{exam}{\rm
Let $\g$ be the Lie algebra with a basis $\{e_{1},e_{1},e_{3}\}$ defined by the nontrivial relations: 
$[e_{1},e_{2}]=e_{2}$ and $[e_{1},e_{3}]=2e_{3}.$ It follows from  \cite{EW06} that 
$\{e_{2},e_{3}\}$ is a basis of $\h:=[\g,\g]$ and $\ad_{e_{1}}$ restricts to an automorphism of $[\g,\g]$.
Let $\sigma\in\aut(\g)$ and $D\in\Der_{\sigma}(\g)$. A direct calculation shows that 
$$\sigma=\begin{pmatrix}
   1 & 0 & 0  \\
   c' & a' & 0 \\
   d' & 0 & b'
\end{pmatrix}$$ for some $a',b'\in\FF^{\times}$ and $c',d'\in \FF$; and 
$$D=\begin{pmatrix}
   0 & 0 & 0  \\
   c & a & 0 \\
   d & 0 & b
\end{pmatrix}$$
for some $a,b,c,d\in\FF$. Note that $\h$ is an abelian Lie algebra, thus 
$$\Der_{\sigma,\h}(\g)=\Der_{\sigma}(\g)=\left\{\begin{pmatrix}
   0 & 0 & 0  \\
   c & a & 0 \\
   d & 0 & b
\end{pmatrix}: a,b,c,d\in\FF\right\}$$ and 
$$\alpha(\Der_{\sigma,\h}(\g))=\left\{\begin{pmatrix}
    a& 0 \\
     0 & b
\end{pmatrix}: a,b\in\FF\right\}.$$
Clearly, the kernel of $\alpha$ is not zero but have dimension at most 2.
\hbo}\end{exam}

\begin{lem}\label{lem4.7}
Let $\sigma\in G$ and $D\in\Der_{\sigma}(\g)$ be arbitrary elements. Then
$$[D,\ad_{x}]=\sigma\circ \ad_{\sigma^{-1}\circ D(x)}$$ for all $x\in\g$.
\end{lem}

\begin{proof}
For all $y\in\g$, it follows that
\begin{eqnarray*}
[D,\ad_{x}](y)&=& (D\circ \ad_{x}-\ad_{x}\circ D)(y)\\
&=& D([x,y])-[x,D(y)]=[D(x),\sigma(y)]\\
&=&\sigma([\sigma^{-1} \circ D(x),y])\\
&=&\sigma\circ \ad_{\sigma^{-1}\circ D(x)} (y).
\end{eqnarray*}
Hence, $[D,\ad_{x}]=\sigma\circ \ad_{\sigma^{-1}\circ D(x)}$.
\end{proof}

Given $x\in\g$ and $\sigma\in G$, we may define a linear map
$$
\varphi_{x}^{\sigma}:\Der_{\sigma}(\g)\ra\ad(\g),\quad
D\mapsto \ad_{\sigma^{-1}\circ D(x)}.
$$
In fact, to see that $\varphi_{x}^{\sigma}$ is linear, for $D,T\in\Der_{\sigma}(\g)$ and $y\in\g$, we have
\begin{eqnarray*}
\varphi_{x}^{\sigma}(D+T)(y)&=&
\ad_{\sigma^{-1}\circ (D+T)(x)}(y)=[\sigma^{-1}\circ (D+T)(x),y]\\
&=&\sigma^{-1}([D(x),\sigma(y)]+[T(x),\sigma(y)])\\
&=& [\sigma^{-1}\circ D(x),y]+[\sigma^{-1}\circ T(x),y]\\
&=&(\varphi_{x}^{\sigma}(D)+\varphi_{x}^{\sigma}(T))(y)
\end{eqnarray*}
which means
that $\varphi_{x}^{\sigma}(D+T)=\varphi_{x}^{\sigma}(D)+\varphi_{x}^{\sigma}(T)$. Similarly, $\varphi_{x}^{\sigma}(a\cdot D)=a\cdot\varphi_{x}^{\sigma}(D)$ for
all $a\in\FF$. Hence, $\varphi_{x}^{\sigma}$ is linear.

The following result describes the kernel of $\varphi_{x}^{\sigma}$.

\begin{prop}\label{prop4.8}
For $x\in\g$ and $\sigma\in G$, we have
$$\ker(\varphi_{x}^{\sigma})=\big\{D \in\Der_{\sigma}(\g)\mid D(x) \in Z(\g)\big\}.$$
Moreover, $\ker(\varphi_{x}^{\sigma})$ is a Lie subalgebra of $\gl(\g)$.
\end{prop}

\begin{proof}
For the first statement, it follows from Lemma \ref{lem4.7} that
\begin{eqnarray*}
\ker(\varphi_{x}^{\sigma}) & = & \{D \in\Der_{\sigma}(\g)\mid \ad_{\sigma^{-1}\circ D(x)}(y)=0,\textrm{ for all }y\in\g\} \\
 & = & \{D \in\Der_{\sigma}(\g)\mid \sigma\circ\ad_{\sigma^{-1}\circ D(x)}(y)=0,\textrm{ for all }y\in\g\} \\
 &=& \{D \in\Der_{\sigma}(\g)\mid [D,\ad_{x}](y)=0,\textrm{ for all }y\in\g\} \\
 &=& \{D \in\Der_{\sigma}(\g)\mid [D(x),\sigma(y)]=0,\textrm{ for all }y\in\g\} \\
 &=& \{D \in\Der_{\sigma}(\g)\mid [D(x),y]=0,\textrm{ for all }y\in\g\}\\
 &=& \{D \in\Der_{\sigma}(\g)\mid D(x) \in Z(\g)\}.
\end{eqnarray*}
Moreover, it is easy to see that $\ker(\varphi_{x}^{\sigma})$ is a vector space. To show $\ker(\varphi_{x}^{\sigma})$ is a Lie algebra,
we take $D,T\in \ker(\varphi_{x}^{\sigma})$ and $y\in\g$. Then
\begin{eqnarray*}
[[D,T](x),\sigma(y)] & = & [D\circ T(x),\sigma(y)] -  [T\circ D(x),\sigma(y)] \\
 & = & D([T(x),\sigma(y)])-[T(x),D(y)]-T([D(x),\sigma(y)])+[D(x),T(y)]\\
 &=& 0.
\end{eqnarray*}
The last equation follows as $T(x)$ and $D(x)$ both are central elements.
Since $\sigma$ is bijective, we see that $[D,T](x)\in Z(\g)$. Hence, $[D,T]\in \ker(\varphi_{x}^{\sigma})$ and
$\ker(\varphi_{x}^{\sigma})$ is a Lie algebra.
\end{proof}

As a corollary, we derive the  third main result. 

\begin{proof}[Proof of Theorem \ref{thm1.5}]
Note that in this case $\ad:\g\ra\ad(\g)$ is an isomorphism and Proposition \ref{prop4.8} implies that
$\varphi_{x_{0}}^{\sigma}$ is injective. Hence, as a subspace, $\Der_{\sigma}(\g)$ can be embedded into $\g$.
\end{proof}

We write $\D_{(\alpha,\beta,\gamma)}(\g)$ for the space of all $(\alpha,\beta,\gamma)$-derivations of $\g$.

\begin{prop}
Assume that $\sigma\in G$ and there exists a scalar  $a\in\FF$ such that $(\sigma-a\cdot I)(\g)$ is
contained in $Z(\g)$. If $a\neq 1$, then $\Der_{\sigma}(\g)=\D_{(1/(a+1),1,0)}(\g)$; if $a=1$, then
$\Der_{\sigma}(\g)=\D_{(1,1,1)}(\g)=\Der(\g)$.
\end{prop}

\begin{proof}
As $(\sigma-a\cdot I)(\g)\subseteq Z(\g)$, it follows that $[D(x),\sigma(y)]=[D(x),ay]$ for all $x,y\in\g$. Then
a linear map $D\in \Der_{\sigma}(\g)\iff D([x,y])=[D(x),\sigma(y)]+[x,D(y)]$ for all $x,y\in\g$ $\iff D([x,y])=a\cdot [D(x),y]+[x,D(y)]$
for all $x,y\in\g$ $\iff D\in \D_{(1,a,1)}(\g)$. Hence, $\Der_{\sigma}(\g)=\D_{(1,a,1)}(\g)$.
If $a-1\neq 0$, then by the second part of the proof of \cite[Theorem 2.2]{NH08}, we see that $\D_{(1,a,1)}(\g)=\D_{(1/(a+1),1,0)}(\g)$. If $a=1$, then $\Der_{\sigma}(\g)=\D_{(1,a,1)}(\g)=\Der(\g)$.
\end{proof}

We close this section with the following description on the order of a periodic $\sigma$-derivation. 

\begin{prop}
Let $\g$ be a nonabelian Lie algebra with a periodic derivation $D\in \Der_{\sigma}(\g)$ of order $m$.
If there exists an eigenvector $v$ of $D$ such that $\sigma(v)=v$, then $m$ is divisible by 6.
\end{prop}

\begin{proof}
Since $D$ is of finite order, it follows that $D$ is diagonalizable.
We suppose $a$ is the eigenvalue of $D$ such that $D(v)=av$. As $\g$ is nonabelian, there exists a nonzero $u\in\g$
such that $[u,v]\neq 0$. We may assume that $u$ is an eigenvector of $D$ corresponding to an eigenvalue $b$.
Then
$$D([u,v])=[D(u),\sigma(v)]+[u,D(v)]=(a+b)[u,v]$$
which means $a+b$ is also an eigenvector of $D$. As $D^{m}=I$, we see that $a^{m}=b^{m}=(a+b)^{m}=1$.
By \cite[Lemma 2.2]{BM12}, we derive that $b=\omega \cdot a$, where $\omega$ denotes a primitive third root of unity.
Thus $a^{m}=\omega^{m}\cdot a^{m}$ and $\omega^{m}=1$. This implies that $m=3m'$ for some $m'\in\N^{+}$.
Moreover, by $1=(a+b)^{m}=(a+\omega \cdot a)^{m}$ we observe that $(1+\omega)^{m}=1$. As $\omega$ is a primitive root,
the fact that $0=\omega^{3}-1=(\omega-1)(\omega^{2}+\omega+1)$ implies that $\omega^{2}+\omega+1=0$, i.e.,
$1+\omega=-\omega^{2}$. Thus
$$1=(1+\omega)^{m}=(-1)^{m}\omega^{2m}=(-1)^{m}(\omega^{m})^{2}=(-1)^{m}.$$
This means that 2 divides $m=3m'$. Hence, $m'$ is an even and $m=6m''$ for some $m''\in\N^{+}$.
\end{proof}

\section{\scshape Computations of $G$-derivations on $\sll_2(\C)$} \label{sec5}
\setcounter{equation}{0}
\renewcommand{\theequation}
{5.\arabic{equation}}
\setcounter{theorem}{0}
\renewcommand{\thetheorem}
{5.\arabic{theorem}}

\noindent To explicitly calculate $G$-derivations, we take a point of view of affine varieties to regard the set of $G$-derivations of $\g$ throughout this closing section.
This viewpoint has been shown to be 
useful to doing specific computations; see \cite{CZ19} and  \cite{GDSSV20}.
Suppose $\dim(\g)=n\in\N^+$. We choose a basis $\{e_1,e_2,\dots,e_n\}$ for $\g$. 
A $G$-derivation $D$ can be determined by finitely many  polynomial equations in at most $3n^2$ variables. 
If we fix $\sigma,\tau\in G$, then an element $D\in \Der_{\sigma, \tau}(\g)$ also  can be determined by finitely many  polynomial equations in at most $n^2$ variables.
This means that $\Der_{G}(\g)$ and $\Der_{\sigma, \tau}(\g)$ both can be viewed as an algebraic variety in some
affine spaces. In general, they are not necessarily irreducible.
Here we consider the case $\g=\sll_2(\C)$ and apply a method from computational ideal theory to calculate $G$-derivations associated with some specific automorphisms of $\sll_2(\C)$.
For the sake of calculation's convenience, we begin with a description of inner automorphisms of $\sll_2(\C)$.

Suppose that $\sll_2(\C)$ is spanned by a basis $\{e_1,e_2,e_3\}$ subject to the following nontrivial relations:
\begin{equation}\label{eq5.1}
[e_1,e_2]=-e_1, [e_1,e_3]=2e_2, [e_2,e_3]=-e_3.
\end{equation}

\begin{thm}\label{thm5.1}
 Let $D\in\Der(\sll_2(\C))$ be a derivation. Then it is of the following form
 $$D=\begin{pmatrix}
    a & b & 0\\
    -2c & 0 & -2b\\
     0 & c & -a\\
\end{pmatrix}$$
 where $a,b,c\in\C$. Moreover, if $bc\neq 0$ and $a^2\neq 4bc$, then $\rank(D^{n})=2$ for all $n\in\N^+$; if $a\neq 0$ and $bc=0$, then $\rank(D^{n})=2$ for all $n\in\N^+$.
\end{thm}

\begin{proof}
A long but direct calculation shows that $D$ has the form. Further, to show the last two statements, we observe that $\det(D)=0$, thus
$\rank(D)\leq 2$. As the submatrix $\bigl(
  \begin{smallmatrix}
    a & b \\
    -2c & 0\\
  \end{smallmatrix}\bigr)$ of $D$ is nonsingular, it follows that $\rank(D)=2$ as $bc\neq 0$. Note that
  $$D^2=\begin{pmatrix}
    a^{2}-2bc & ab & -2b^{2}\\
    -2ac & -4bc & 2ab\\
     -2c^{2} & -ac & a^{2}-2bc\\
\end{pmatrix}$$
has determinant zero and a nonsingular submatrix $\bigl(
  \begin{smallmatrix}
    a^{2}-2bc & ab  \\
    -2ac & -4bc\\
  \end{smallmatrix}\bigr)$, we see that $\rank(D)=2=\rank(D^2)$. Now the second statement follows
  from the well-known fact in Linear Algebra that if there exists $m\in\N^+$ such that $\rank(D^m)=\rank(D^{m+1})$, then $\rank(D^m)=\rank(D^{m+j})$ for all $j\in\N^+$. To prove the third one, we may suppose $b=0$. Hence,
   $$D^n=\begin{pmatrix}
    a^{n} & 0 & 0\\
    \ast & 0& 0\\
     \ast & \ast & (-1)^na^{n}\\
\end{pmatrix}
$$
 has rank 2 for all $n\in\N^+$.
\end{proof}

\begin{thm}\label{thm5.2}
Let $D=\left(\begin{smallmatrix}
    a & b & 0\\
    -2c & 0 & -2b\\
     0 & c & -a\\
\end{smallmatrix}\right)
\in\Der(\sll_2(\C))$ be a derivation. Then $D$ is nilpotent if and only if it satisfies one of the following two cases: {\rm (1)} $bc=0$ and $a=0$; {\rm (2)} $bc\neq 0$ and $a^2=4bc$.
\end{thm}

\begin{proof}
We observe that when $bc=0$ and $a=0$ or when $bc\neq 0$ and $a^2=4bc$, it follows that $D^3=0$. Thus $D$ is nilpotent.  Conversely, suppose $D$ is nilpotent.
By the last two assertions  in Theorem \ref{thm5.1}, we see that $a^2=4bc$ if $bc\neq 0$; and $a=0$ if $bc=0$.
\end{proof}

By the definition of the group of inner automorphisms, we obtain the following direct but important description on
inner automorphisms of $\sll_2(\C)$.

\begin{thm}\label{thm5.3}
Let
$$D_b:=
  \begin{pmatrix}
    0 & b & 0\\
    0 & 0 & -2b\\
     0 & 0 & 0\\
  \end{pmatrix}, D_c:=
  \begin{pmatrix}
    0 & 0 & 0\\
    -2c & 0 & 0\\
     0 & c & 0\\
  \end{pmatrix} \textrm{ and } D_{a,b}=
  \begin{pmatrix}
    a & b & 0\\
    -\frac{a^2}{2b} & 0 & -2b\\
     0 & \frac{a^2}{4b} & -a\\
  \end{pmatrix} (a\neq 0).$$
Then $\iaut(\sll_2(\C))$ is generated by $\{\exp(D_b),\exp(D_c),\exp(D_{a,b})\mid a,b,c\in\C\}$, where
$$\exp(D_b)=
  \begin{pmatrix}
    1 & b & -b^2\\
    0 & 1 & -2b\\
     0 & 0 & 1\\
  \end{pmatrix},\exp(D_c)=
  \begin{pmatrix}
    1 & 0 & 0\\
    -2c & 1 & 0\\
     -c^2 & c & 1\\
  \end{pmatrix}\textrm{ and } \exp(D_{a,b})=I_3+D_{a,b}+\frac{D_{a,b}^2}{2}.$$
\end{thm}

\begin{proof}
Note that every derivation of $\sll_2(\C)$ is inner and the general fact that an inner automorphism of a Lie algebra can be generated by the image of nilpotent inner derivations under the exponential map $\exp$. This statement follows directly from Theorems \ref{thm5.1} and \ref{thm5.2}.
\end{proof}

The rest of this section is to calculate $\Der_{\sigma}(\sll_2(\C))$ and describe its irreducible components  when $\sigma$ were taken to be one of these generators of  $\iaut(\sll_2(\C))$. We will take $\sigma=\exp(D_b)$ as a sample to 
explain our ideas and techniques. Similar results for $\sigma=\exp(D_c)$ and $\exp(D_{a,b})$ will be summarized at the end of this section without giving detailed proofs.

From now on we let $\sigma=\exp(D_b)$ and $D\in \Der_{\sigma}(\sll_2(\C))$ be an arbitrary elements. Note that $\sigma$ is 
parameterized by $b$, thus in what follows we view $\Der_{\sigma}(\sll_2(\C))$ as the affine variety consisting of all pairs $(D,\sigma)$ where $D$ is a $(\sigma,1)$-derivation.

With respect to the basis $\{e_1,e_2,e_3\}$, we may suppose $D=(a_{ij})_{3\times 3}^t$. Thus
\begin{eqnarray}
(D(e_1),D(e_2),D(e_3))&=&(e_1,e_2,e_3)
\begin{pmatrix}
a_{11} & a_{21} & a_{31} \\
a_{12} & a_{22} & a_{32} \\
a_{13} & a_{23} & a_{33} \\
\end{pmatrix} \label{eq5.2}\\
(\sigma(e_1),\sigma(e_2),\sigma(e_3))&=&(e_1,e_2,e_3)
  \begin{pmatrix}
    1 & b & -b^2\\
    0 & 1 & -2b\\
     0 & 0 & 1\\
  \end{pmatrix}.\label{eq5.3}
\end{eqnarray}

Let $A=\C[x_{ij},y\mid 1\leq i,j\leq 3]$ be the polynomial ring in 10 variables. We define
\begin{eqnarray*}
   f_1:= x_{11}+x_{33} && f_2:= x_{12}+2x_{23}\\
   f_3:= x_{21}+\frac{1}{2}x_{32}&&f_4:= x_{23}y\\
   f_5:= x_{31}-\frac{1}{2}x_{32}y&& f_6:= x_{33}y.
\end{eqnarray*}
and $J$ to be the ideal generated by $\B:=\{x_{13},f_i\mid 1\leq i\leq 6\}$ in $A$. We use the lexicographical order with $x_{11}>x_{12}>x_{13}>x_{21}>x_{22}>x_{23}>x_{31}>x_{32}>x_{33}>y$.

\begin{lem}\label{lem5.4}
$\Der_{\sigma}(\sll_2(\C))\subseteq \V(J)$.
\end{lem}

\begin{proof}
As  $D=(a_{ij})_{3\times 3}^t \in \Der_{\sigma}(\sll_2(\C))$, we have
$D([e_{i},e_{j}])=[D(e_{i}),\sigma(e_{j})]+[e_{i},D(e_{j})]$ for all $i,j\in \{1,2,3\}$.
It follows from (\ref{eq5.1}), (\ref{eq5.2}) and (\ref{eq5.3}) that
\begin{eqnarray*}
    2a_{21}b-a_{22}b^{2}+2a_{31}=0 &&
    a_{12}+2a_{13}b+2a_{23}=0\\
    2a_{21}+2a_{23}b^{2}+a_{32}=0&&
    a_{13}=0\\
    2a_{21}+a_{32}=0&&
    a_{12}b-a_{22}=0\\
    a_{12}+2a_{23}=0&&
    2a_{11}b-a_{12}b^{2}-2a_{21}-a_{32}=0\\
    -a_{11}+a_{22}-a_{33}=0&&
    a_{12}+2a_{23}=0\\
    a_{12}-2a_{13}b+2a_{23}=0&&
    a_{11}+a_{13}b^{2}-a_{22}+a_{33}=0\\
    a_{22}+2a_{23}b=0&&
    a_{22}=0\\
    -2a_{31}+a_{32}b=0&&
    2a_{21}+a_{32}+2a_{33}b=0.
\end{eqnarray*}
Simplifying these equations, we see that
\begin{eqnarray*}
    a_{11}+a_{33}=0 &&
    a_{12}+2a_{23}=0\\
    a_{13}=0&&
    a_{21}+\frac{1}{2}a_{32}=0\\
    a_{22}=0&&
    a_{23}b=0\\
    a_{31}-\frac{1}{2}a_{32}b=0&&
    a_{33}b=0.
\end{eqnarray*}
Hence, $(D,\sigma)\in \V(J)$, i.e., $\Der_{\sigma}(\sll_2(\C))\subseteq\V(J)$.
\end{proof}

\begin{prop}\label{prop5.5}
Let $\p_1$ be the ideal of $A$ generated by $\B_1:=\{x_{13},x_{22},x_{31},y,f_i\mid 1\leq i\leq 3\}$ and $\p_2$ be the ideal of $A$ generated by $\B_2:=\{x_{11},x_{12},x_{13},x_{22},x_{23},x_{33},f_{3},f_{5}\}$. Then $\p_1$ and $\p_2$ are prime ideals.
\end{prop}

\begin{proof} To show $\p_1$ and $\p_{2}$ are prime ideals, it suffices to show that the quotient rings $A/\p_1$ and $A/\p_{2}$ are integral domains. Indeed, by the definition of elements in $\B_1$ and $\B_{2}$, we see that $A/\p_1\cong \C[x_{11},x_{12},x_{21}]$ and
$A/\p_2\cong \C[x_{32},y]$. They both are integral domains. Hence, $\p_1$ and $\p_2$ are prime.
\end{proof}

\begin{coro}\label{coro5.6}
The irreducible affine variety $\V(\p_{1})=\Der(\sll_{2}(\C))$ is three-dimensional and  $\V(\p_{2})$ is a two-dimensional irreducible  affine variety, consisting of all paris $(D,\sigma)$ of the following form:
$$D=\begin{pmatrix}
  0    &  \frac{-a}{2} &\frac{ab}{2} \\
    0  & 0&a\\
    0&0&0
\end{pmatrix}\textrm{ and }\sigma=\begin{pmatrix}
  1 & b & -b^2\\
    0 & 1 & -2b\\
     0 & 0 & 1\\
\end{pmatrix}.$$
\end{coro}

\begin{lem}\label{lem5.7}
 $V(\p_1)\cup V(\p_2)\subseteq V(J)$.
\end{lem}

\begin{proof}
We need to show that $J \subseteq \p_1$ and $J \subseteq \p_2$. As $x_{13},f_{1},f_{2},f_{3}\in\B_{1}$, it suffices to show that $f_{4},f_{5},f_{6}$ are in $\p_1$.  Note that $x_{31},y$ belong to $\p_{1}$, thus $f_{4},f_{5},f_{6}\in\p_1$.
Similarly, to show $J \subseteq \p_2$, we only need to show $f_{1},f_{2},f_{4}$ and $f_{6}$ belong to $\p_{2}$.
This is easy to verify.
\end{proof}

\begin{lem}\label{lem5.8}
 $V(J) \subseteq V(\p_1)\cup V(\p_2)$.
\end{lem}

\begin{proof}
 Since $V(\p_1)\cup V(\p_2)=V(\p_1\cdot \p_2)$, it suffices to show that $\p_1\cdot \p_2$ is contained in $J$. By
 \cite[Proposition 6, p.185]{CLO07}, we need only to prove that all elements in $\B_1\cdot \B_2:=\{b_1b_2\mid b_1\in \B_1,b_2\in\B_2\}$ belong to $J$, where $\B_1$ and $\B_2$ are defined as in Proposition \ref{prop5.5}.
We observe that $f_{1},f_{2},f_{3},f_{5},x_{22}$ and $x_{13}\in J$, thus $\B_1 $and $\B_2$ can be replaced by $\B_1'$ and $\B_2'$ respectively, where
\begin{eqnarray*}
\B_1'&:=&\{x_{31},y\}  \\
\B_2'&:=&\{x_{11},x_{12},x_{23},x_{33}\}
\end{eqnarray*}
Now we need to show that each element of $\B_1'\cdot\B_2'$ belongs to $J$. To do this, we are working over modulo $J$.
As $x_{11}\equiv -x_{33}$ and $x_{31}\equiv\frac{1}{2}x_{32}y$, we see that
$x_{11}x_{31}\equiv -\frac{1}{2}x_{32}x_{33}y=-\frac{1}{2}x_{32}f_6\equiv 0.$
As $x_{12}\equiv -2x_{23}$, it follows that $x_{12}x_{31}\equiv-x_{32}(x_{23}y)=-x_{32}f_4\equiv 0.$
Similarly, $x_{31}x_{23}\equiv\frac{1}{2}x_{32}f_4\equiv 0$ and
$x_{31}x_{33}\equiv\frac{1}{2}x_{32}f_6\equiv 0$.
Moreover, $yx_{11}\equiv -f_6\equiv 0, yx_{12}\equiv-2f_4\equiv0, yx_{23}=f_4\in J$ and $yx_{33}=f_6\in J$. This shows that $\B_1\cdot\B_2\subseteq J$.
Hence, $V(J)\subseteq V(\p_1)\cup V(\p_2)$, as desired.
\end{proof}

Combining Lemmas \ref{lem5.7} and \ref{lem5.8}, we derive

\begin{coro}\label{coro5.9}
 $V(J)=V(\p_1)\cup V(\p_2)$.
\end{coro}

Now we are in a position to prove the last main result, i.e., 
$\Der_{\sigma}(\sll_2(\C))=V(\p_1)\cup V(\p_2)$.

\begin{proof}[Proof of Theorem \ref{thm1.6}]
By Lemma \ref{lem5.4} and Corollary \ref{coro5.9}, it suffices to show that $V(\p_1)$ and $ V(\p_2)$ both are contained in $\Der_{\sigma}(\sll_2(\C))$. Corollary \ref{coro5.6} implies that $\V(\p_{1})=\Der(\sll_{2}(\C))\subseteq \Der_{\sigma}(\sll_2(\C))$. To see that $V(\p_2)\subseteq \Der_{\sigma}(\sll_2(\C))$, by Corollary \ref{coro5.6}, we may take an arbitrary  element $$D=
  \begin{pmatrix}
  0  &  \frac{-a}{2} & \frac{ab}{2} \\
  0  &     0         &  a\\
  0  &     0         &  0\\
  \end{pmatrix}\in V(\p_{2}).$$
  A direct calculation shows that
\begin{eqnarray*}
D([e_{1},e_{2}])=0=-D(e_{1})&&
D([e_{2},e_{1}])=0=D(e_{1})\\
D([e_{1},e_{3}])=-ae_{2}=2D(e_{2})&&
D([e_{3},e_{1}])=ae_{2}=-2D(e_{2})\\
D([e_{2},e_{3}])=-\frac{ab}{2}e_{1}-ae_{2}=-D(e_{3})&&
D([e_{2},e_{3}])=\frac{ab}{2}e_{1}+ae_{2}=D(e_{3})
\end{eqnarray*}
which means that $D([e_{i},e_{j}])=[D(e_{i}),\sigma(e_{j})]+[e_{i},D(e_{j})]$ for all $i,j\in \{1,2,3\}$. Hence, $(D,\sigma)\in \Der_{\sigma}(\sll_2(\C))$ and $V(\p_2)\subseteq \Der_{\sigma}(\sll_2(\C))$.
Therefore, $\Der_{\sigma}(\sll_2(\C))=V(\p_1)\cup V(\p_2)$.
\end{proof}

\begin{coro}
If we fix $b$ in $\sigma$, then as a $\C$-vector space, $\dim(\Der_{\sigma}(\sll_2(\C)))=4$.
\end{coro}

Similar methods can be applied to the the cases where $\sigma=\exp(D_{c})$ and $\sigma=\exp(D_{a,b})$.

\begin{thm}\label{thm5.11}
Let $\sigma=\exp(D_{c})$. Then  $\Der_{\sigma}(\sll_2(\C))$ can be decomposed into two irreducible components
$\V(\p_{1})$ and $\V(\p_{2})$, where $\V(\p_{1})=\Der(\sll_{2}(\C))$ and  $\V(\p_{2})$ is a two-dimensional, consisting of all paris $(D,\sigma)$ of the following form:
$$D=\begin{pmatrix}
  0    &  0 &0 \\
  -2a  &  0 &a\\
  -2ac &  0 &0
\end{pmatrix}\textrm{ and }\sigma=\begin{pmatrix}
  1 & 0 & 0\\
    -2c & 1 & 0\\
     -c^{2} & c & 1\\
\end{pmatrix}.$$
In particular,  if we fix $c$ in $\sigma$, then $\Der_{\sigma}(\sll_2(\C))$ is a four-dimensional vector space over $\C$.
\end{thm}

\begin{thm}\label{thm5.12}
Let $\sigma=\exp(D_{a,b})$. Then  $\Der_{\sigma}(\sll_2(\C))$ can be decomposed into two irreducible components
$\V(\p_{1})$ and $\V(\p_{2})$, where $\V(\p_{1})=\Der(\sll_{2}(\C))$ and  $\V(\p_{2})$ is a three-dimensional, consisting of all paris $(D,\sigma)$ of the following form:
$$D=\begin{pmatrix}
 \frac{ab+4}{ab-4}\cdot c& \frac{-a^2b - 2a}{ab^2 - 4b}\cdot c &\frac{-a^{3}/4}{ab^{2}-4b}\cdot c\\
 \frac{2ab^2 + 4b}{a^{2}b-4a}\cdot c& \frac{-2ab}{ab-4}\cdot c & \frac{a-a^2b/2}{ab^2 - 4b}\cdot c\\
 \frac{-4b^3}{a^2b - 4a}\cdot c&  \frac{4ab^2 - 8b}{a^2b - 4a} \cdot c&c
\end{pmatrix}\textrm{ and }\sigma=\exp(D_{a,b}).$$
In particular,  if we fix $a,b$ in $\sigma$, then $\Der_{\sigma}(\sll_2(\C))$ is a four-dimensional vector space over $\C$.
\end{thm}

\begin{rem}{\rm
Compared with \cite{CZ19} in which the similar techniques have been applied to describe Hom-Lie algebra structures on
some important Lie algebras such as the general linear Lie algebra and the Heisenberg Lie algebra, our method in this section might have potential capability to study other linear objects determined by polynomials equations.
\hbo}\end{rem}


\raggedright

\begin{bibdiv}
  \begin{biblist}
  

\bib{BBC99}{article}{
   author={Beidar, K. I.},
   author={Bresar, Matej},
   author={Chebotar, M. A.},
   title={Generalized functional identities with (anti-) automorphisms and
   derivations on prime rings. I},
   journal={J. Algebra},
   volume={215},
   date={1999},
   number={2},
   pages={644--665},
}

\bib{BG97}{article}{
   author={Bergen, Jeffrey},
   author={Grzeszczuk, Piotr},
   title={Invariants of skew derivations},
   journal={Proc. Amer. Math. Soc.},
   volume={125},
   date={1997},
   number={12},
   pages={3481--3488},
}

\bib{BK89}{article}{
   author={Bell, H. E.},
   author={Kappe, L.-C.},
   title={Rings in which derivations satisfy certain algebraic conditions},
   journal={Acta Math. Hungar.},
   volume={53},
   date={1989},
   number={3-4},
   pages={339--346},
}

\bib{BKK95}{article}{
   author={Benkart, Georgia},
   author={Kostrikin, Alexei I.},
   author={Kuznetsov, Michael I.},
   title={Finite-dimensional simple Lie algebras with a nonsingular
   derivation},
   journal={J. Algebra},
   volume={171},
   date={1995},
   number={3},
   pages={894--916},
}

\bib{BM12}{article}{
   author={Burde, Dietrich},
   author={Moens, Wolfgang Alexander},
   title={Periodic derivations and prederivations of Lie algebras},
   journal={J. Algebra},
   volume={357},
   date={2012},
   pages={208--221},
}




\bib{Bre93}{article}{
   author={Bresar, Matej},
   title={Centralizing mappings and derivations in prime rings},
   journal={J. Algebra},
   volume={156},
   date={1993},
   number={2},
   pages={385--394},
}


\bib{Bre08}{article}{
   author={Bresar, Matej},
   title={Near-derivations in Lie algebras},
   journal={J. Algebra},
   volume={320},
   date={2008},
   number={10},
   pages={3765--3772},
}

\bib{Bur06}{article}{
   author={Burde, Dietrich},
   title={Left-symmetric algebras, or pre-Lie algebras in geometry and
   physics},
   journal={Cent. Eur. J. Math.},
   volume={4},
   date={2006},
   number={3},
   pages={323--357},
}

\bib{CGKS02}{collection}{
   title={Differential algebra and related topics},
   booktitle={Proceedings of the International Workshop held at Rutgers
   University, Newark, NJ, November 2--3, 2000},
   editor={Cassidy, Phyllis J.},
   editor={Guo, Li},
   editor={Keigher, William F.},
   editor={Sit, William Y.},
   publisher={World Scientific Publishing Co., Inc., River Edge, NJ},
   date={2002},
   pages={xiv+305},
   isbn={981-02-4703-6},
}



\bib{CL05}{article}{
   author={Chuang, Chen-Lian},
   author={Lee, Tsiu-Kwen},
   title={Identities with a single skew derivation},
   journal={J. Algebra},
   volume={288},
   date={2005},
   number={1},
   pages={59--77},
}




\bib{CLO07}{book}{
   author={Cox, David},
   author={Little, John},
   author={O'Shea, Donal},
   title={Ideals, varieties, and algorithms},
   series={Undergraduate Texts in Mathematics},
   edition={3},
   publisher={Springer, New York},
   date={2007},
   pages={xvi+551},
}

\bib{CZ19}{article}{
   author={Chen, Yin},
   author={Zhang, Runxuan},
   title={A commutative algebra approach to multiplicative Hom-Lie algebras},
   journal={arXiv:1907.02415},
   date={2019},
}

\bib{DFW18}{article}{
   author={De Filippis, Vincenzo},
   author={Wei, Feng},
   title={$b$-generalized $(\alpha,\beta)$-derivations and $b$-generalized
   $(\alpha,\beta)$-bideri- vations of prime rings},
   journal={Taiwanese J. Math.},
   volume={22},
   date={2018},
   number={2},
   pages={313--323},
}

\bib{EW06}{book}{
   author={Erdmann, Karin},
   author={Wildon, Mark J.},
   title={Introduction to Lie algebras},
   series={Springer Undergraduate Mathematics Series},
   publisher={Springer-Verlag London, Ltd., London},
   date={2006},
   pages={x+251},
}

\bib{GDSSV20}{article}{
   author={Garcia-Delgado, R.},
   author={Salgado, G.},
   author={Sanchez-Valenzuela, O. A.},
   title={On 3-dimensional complex Hom-Lie algebras},
   journal={J. Algebra},
   volume={555},
   date={2020},
   pages={361--385},
}

\bib{GK08}{article}{
   author={Guo, Li},
   author={Keigher, William},
   title={On differential Rota-Baxter algebras},
   journal={J. Pure Appl. Algebra},
   volume={212},
   date={2008},
   number={3},
   pages={522--540},
   issn={0022-4049},
}



\bib{GL14}{article}{
   author={Guo, Li},
   author={Li, Fang},
   title={Structure of Hochschild cohomology of path algebras and
   differential formulation of Euler's polyhedron formula},
   journal={Asian J. Math.},
   volume={18},
   date={2014},
   number={3},
   pages={545--572},
   issn={1093-6106},
}


\bib{HLS06}{article}{
   author={Hartwig, Jonas T.},
   author={Larsson, Daniel},
   author={Silvestrov, Sergei D.},
   title={Deformations of Lie algebras using $\sigma$-derivations},
   journal={J. Algebra},
   volume={295},
   date={2006},
   number={2},
   pages={314--361},
}

\bib{Hva98}{article}{
   author={Hvala, Bojan},
   title={Generalized derivations in rings},
   journal={Comm. Algebra},
   volume={26},
   date={1998},
   number={4},
   pages={1147--1166},
}

\bib{Jac55}{article}{
   author={Jacobson, N.},
   title={A note on automorphisms and derivations of Lie algebras},
   journal={Proc. Amer. Math. Soc.},
   volume={6},
   date={1955},
   pages={281--283},
}

\bib{Kol73}{book}{
   author={Kolchin, E. R.},
   title={Differential algebra and algebraic groups},
   publisher={Academic Press, New York-London},
   date={1973},
   pages={xviii+446},
}





\bib{KP92}{article}{
   author={Kharchenko, V. K.},
   author={Popov, A. Z.},
   title={Skew derivations of prime rings},
   journal={Comm. Algebra},
   volume={20},
   date={1992},
   number={11},
   pages={3321--3345},
}

\bib{LL00}{article}{
   author={Leger, George F.},
   author={Luks, Eugene M.},
   title={Generalized derivations of Lie algebras},
   journal={J. Algebra},
   volume={228},
   date={2000},
   number={1},
   pages={165--203},
}



\bib{NH08}{article}{
   author={Novotny, Petr},
   author={Hrivnak, Jiri},
   title={On $(\alpha,\beta,\gamma)$-derivations of Lie algebras and
   corresponding invariant functions},
   journal={J. Geom. Phys.},
   volume={58},
   date={2008},
   number={2},
   pages={208--217},
}

\bib{Pos57}{article}{
   author={Posner, Edward C.},
   title={Derivations in prime rings},
   journal={Proc. Amer. Math. Soc.},
   volume={8},
   date={1957},
   pages={1093--1100},
}

\bib{Rit32}{book}{
   author={Ritt, Joseph F.},
   title={Differential equations from the algebraic standpoint},
   series={American Mathematical Society Colloquium Publications},
   volume={14},
   publisher={American Mathematical Society, New York},
   date={1932},
   pages={x+172},
}

\bib{vdPS03}{book}{
   author={van der Put, Marius},
   author={Singer, Michael F.},
   title={Galois theory of linear differential equations},
   series={Grundlehren der Mathematischen Wissenschaften},
   volume={328},
   publisher={Springer-Verlag, Berlin},
   date={2003},
   pages={xviii+438},
   isbn={3-540-44228-6},
}


\bib{Zus10}{article}{
   author={Zusmanovich, Pasha},
   title={On $\delta$-derivations of Lie algebras and superalgebras},
   journal={J. Algebra},
   volume={324},
   date={2010},
   number={12},
   pages={3470--3486},
}


  \end{biblist}
\end{bibdiv}
\end{document}